\documentclass[12pt]{article}
 
\usepackage[margin=1in]{geometry} 
\usepackage{amsmath,amsthm,amssymb,mathrsfs}
\theoremstyle{plain}
\newtheorem{theorem}{Theorem}[section]
\newtheorem{corollary}[theorem]{Corollary}
\newtheorem{lemma}[theorem]{Lemma}
\newtheorem{proposition}[theorem]{Proposition}
\theoremstyle{definition}
\newtheorem{definition}{Definition}[section]
\theoremstyle{remark}

\usepackage{graphicx}
\usepackage{indentfirst}
\usepackage{tikz-cd}
\usepackage{mathtools}

\begin{document}
\title{Optimal embeddings for maximal orders of central simple algebras of degree 3 over number fields} 
\author{Yuxuan Yang} 
\maketitle
\begin{abstract}
Let  $B$ be a central simple algebra of degree 3 over a number field $F$ and $K/F$ be a finite extension of degree 3. For an order $S$ of $K$, we determine exactly when $S$ cannot be optimally embedded into all maximal orders of $B$. Moreover, we further determine exactly when $S$ can be optimally embedded into $\frac{1}{3}$ isomorphism classes of maximal orders of $B$ and $\frac{2}{3}$ isomorphism classes of maximal orders of $B$ in the rest of cases.  
\end{abstract}
\section{Introduction}
It is well-known that the local-global principle holds for embeddings of a field extension into a central simple algebra over a number field by class field theory. However, the local-global principle is no longer true for integral embeddings. Indeed, if $F$ is a number field with the ring of integers $R$ and $K/F$ is a finite extension of degree $n$, Chevalley \cite{Chevalley1936} first proved that the ratio of the number of isomorphism classes of maximal orders in the $n\times n$ matrix algebra $M_n(F)$ over $F$ where the ring of integers $\mathcal{O}_K$ of $K$ can be embedded to the total number of isomorphism classes of maximal orders is $[H_F\cap K: F]^{-1}$ where $H_F$ is the Hilbert class field of $F$. Later, Chinburg and Friedman \cite{chinburg1999embedding} studied integral embeddings of quaternion algebras for maximal orders. They call an order selective if it cannot be embedded into all maximal orders of the quaternion algebra. Chan and Xu \cite{chan2004representations} and Guo and Qin \cite{guo2004embedding} extended their result to Eichler orders. In their series papers \cite{Linowitz2012} and \cite{LinowitzShemanske2017}, Linowitz and Shemanske further studied the integral embeddings for central simple algebras of higher degree. 

 The notion of optimal embeddings was first introduced by Eichler \cite{Eichler1955}, who also provided criteria for determining the existence of optimal embeddings into square-free level Eichler orders in quaternion algebras. Maclachlan \cite{MACLACHLAN20082852} later established the optimal selectivity results for square-free level Eichler orders in quaternion algebras. Subsequently, Voight \cite{voight2021quaternion} extended the optimal selectivity to arbitrary Eichler orders. Xue and Yu \cite{XUE2024166} further generalized Voight’s result to arbitrary orders in quaternion algebras.

In this paper, we study the optimal embeddings into maximal orders of central simple algebras of degree $3$ and establish the selectivity condition for such embeddings.  

Let $B$ be a central simple algebra of degree $3$ over a number field $F$, let $R$ be the ring of integers of $F$, and let $O$ be a maximal order of $B$. Denote by $\mathrm{Gen}(O)$ the set of all maximal orders of $B$. Let $K/F$ be a field extension of degree $3$ inside $B$, and let $H$ be the Hilbert class field of $F$. When $K\subseteq H$, we have the restriction map
$$\rho: \mathrm{Cl}(R)\cong \mathrm{Gal}(H/F)\longrightarrow \mathrm{Gal}(K/F).$$
For any $R$-order $S$ in $K$, let $\mathrm{disc}(S)$ be the discriminant ideal of $S$ over $R$. If $K/F$ is everywhere unramified, $\mathrm{disc}(S)$ is the square of an ideal in $R$, and we denote this ideal by $\sqrt{\mathrm{disc}(S)}$. Define the selectivity set of $S$ by
$$D(S) = \{ \rho([\mathfrak{p}^k]) : \mathfrak{p} \text{ is a prime ideal of } R,\ \mathfrak{p}^k \mid \sqrt{\mathrm{disc}(S)},\ \mathfrak{p}^{k+1} \nmid \sqrt{\mathrm{disc}(S)} \}.$$
The main result of this paper is the following theorem.

\begin{theorem}[Theorem \ref{main}]
Let $K/F$ be a field extension of degree $3$ inside $B$. Suppose that $S$ is an $R$-order of $K$ with an optimal embedding $S\hookrightarrow O$. Then $S$ cannot be optimally embedded into every maximal order of $B$ if and only if 
$$\begin{cases} 
K/F \text{ is everywhere unramified}, \\
B\cong \mathrm{M}_3(F), \\
D(S)\neq \mathrm{Gal}(K/F).
\end{cases}$$
In this situation, $S$ can be optimally embedded into exactly $\dfrac{|D(S)|}{3}$ of the isomorphism classes in $\mathrm{Gen}(O)$.
\end{theorem}

Notation and terminology are standard if not explained; some of them are adopted from \cite{voight2021quaternion}. Let $F$ be a number field. We use $\underline{F}^\times$ to denote the idele group of $F$. If $M$ is a finite dimensional $F$-vector space or a finitely generated $R$-module, then $M_\mathfrak{p}$ stands for the completion of $M$ at the prime $\mathfrak{p}$ of $F$. Write $\widehat{\mathbb{Z}}=\prod_p \mathbb{Z}_p$ as the profinite completion of $\mathbb{Z}$. If $X$ is a finite dimensional $\mathbb{Q}$-vector space or a finitely generated $\mathbb{Z}$-module, we define $\widehat{X}=X\otimes_\mathbb{Z}\widehat{\mathbb{Z}}$.

\section{Optimal embeddings and selectivity sandwich}
In this section, we always assume that $B$ is a central simple algebra of a prime degree $p>2$. First, we give the definition of an optimal embedding.

\begin{definition}
Let $B$ be a central simple algebra of degree $p$ over a number field $F$ and $K/F$ be a field extension of degree $p$ inside $B$ and $R$ be the ring of integers of $F$. Suppose $S$ is a $R$-order of $S$ and $O$ is a $R$-order of $B$. An embedding of $R$-algebra $\varphi:S\to O$ is called  optimal if $\varphi(K)\cap O=\varphi(S)$.
\end{definition}
Let 
$\mathrm{Emb}_R(S,O)$ be the set of all optimal embeddings from $S$ to $O$. 
We simply write $\mathrm{Emb}(S,O)$ if there is no confusion.
By the Skolem–Noether Theorem (see \cite[Main Theorem 7.7.1]{voight2021quaternion}), we have $$K^\times\backslash E\cong \mathrm{Emb}(S,O)$$ where 
$$E:=\{\beta\in B^\times\mid \beta^{-1}K\beta\cap O= \beta^{-1}S\beta\}. $$
In particular, there are optimal embeddings from $S$ to $O$ if and only if $E\neq \varnothing$. 

Let 
$$N_{B^\times}(O):=\{\beta\in B^\times\mid \beta^{-1}O\beta=O\}$$ and 
$O^\times\leq \Gamma\leq N_{B^\times}(O)$. Then $\Gamma$ acts on the right on $\mathrm{Emb}(S,O)$ by conjugation. That is, for $\gamma \in \Gamma$ and $\phi \in \mathrm{Emb}(S,O)$, we define
$$(\phi\gamma)(\alpha):=\gamma^{-1}\phi(\alpha)\gamma . $$
Let
$$\mathrm{Emb}(S,O;\Gamma):=\{\Gamma\text{-conjugacy classes of optimal embeddings } S\hookrightarrow O\}$$
and 
$$m(S,O;\Gamma):=\#\mathrm{Emb}(S,O;\Gamma). $$
Then $$K^\times\backslash E/\Gamma\cong \mathrm{Emb}(S,O;\Gamma) . $$

We say that two orders $O$ and $O'$ of $B$ are locally isomorphic if there exists an element $\widehat{v}\in \widehat{B}^\times$ such that $\widehat{O'}=\widehat{v}\widehat{O}\widehat{v}^{-1}$. 
Define the \textit{genus} $\mathrm{Gen} (O)$ as the set of $R$-orders in $B$ that are locally isomorphic to $O$. The \textit{type} $\mathrm{Typ} (O)$ of $O$ is defined to be the set of all isomorphism classes of orders in $\mathrm{Gen} (O)$. Then there are natural bijections 
$$\mathrm{Gen} (O) \cong \widehat{B}^\times/N_{\widehat{B}^\times}(\widehat{O}) \ \ \ \text{and} \ \ \ \mathrm{Typ} (O) \cong B^\times\backslash\widehat{B}^\times/N_{\widehat{B}^\times}(\widehat{O})$$ where 
$$ N_{\widehat{B}^\times}(\widehat{O})= \{\widehat{v} \in  \widehat{B}^\times\mid\widehat{O}=\widehat{v}^{-1}\widehat{O}\widehat{v} \} . $$

\begin{theorem}[{\cite[Theorem 3.1]{Linowitz2012}}]\label{NR}
The reduced norm on B induces a bijection:
$$N_r:\mathrm{Typ} (O)\cong B^\times\backslash\widehat{B}^\times/N_{\widehat{B}^\times}(\widehat{O})\xrightarrow{\cong} F^\times\backslash \widehat{F}^\times/N_r(N_{\widehat{B}^\times}(\widehat{O}))$$ where $N_r$ is the reduced norm of $B$ over $F$.
\end{theorem}
\begin{proof} The proof in \cite[Theorem 3.1]{Linowitz2012} is valid for any order $O$ of $B$ (see also \cite[Proposition 2.7]{xie2025integralembeddingscentralsimple}).
\end{proof}

In order to apply the Artin map in class field theory, we need the following lemma.

\begin{lemma}\label{ext}
If $K/F$ is a finite extension of number fields, then inclusion $\widehat{F}^\times \subset\underline{F}^\times$ induces an isomorphism $$\widehat{F}^\times/F^\times_{K} N_m(\widehat{K}^\times)\xrightarrow{\cong} \underline{F}^\times/F^\times N_m(\underline{K}^\times)$$
where $$F^\times_{K} =\{\alpha \in F^\times\mid  \alpha_\mathfrak{p}>0 \text{ for all real place } \mathfrak{p} \text{ ramified in }K\}$$ and $N_m$ is the norm map.
\end{lemma}

\begin{proof} Since the homomorphism
$$ \widehat{F}^\times\longrightarrow \underline{F}^\times/F^\times N_m(\underline{K}^\times)$$ induced by inclusion $\widehat{F}^\times \subset\underline{F}^\times$ has the kernel $F^\times_{K} N_m(\widehat{K})$ by a direct computation, one only needs to show the surjectivity which follows from weak approximation.
\end{proof}

\begin{corollary}
There is one to one correspondence 
\begin{align*}
\{K/F \text{ finite abelian}\}&\longleftrightarrow \{H\leq \widehat{F}^\times \text{ open subgroup with finite index containing }F^\times_{>0}\}\\
K &\to F^\times_{K} N_m(\widehat{K}^\times)
\end{align*}
where $F^\times_{>0}=\{ \alpha\in F^\times: \ \alpha>0 \ \text{in $F_\mathfrak{p}$ for all real primes $\mathfrak{p}$ of $F$} \}$. 
\end{corollary}
\begin{proof} It follows from Lemma \ref{ext} and the class field theory.
\end{proof}
Define $GN(O)=F^\times N_r(N_{\widehat{B}^\times}(\widehat{O}))$. Then $GN(O)$ is an open subgroup of finite index in $\widehat{F}^\times$ and corresponds to a class field $H_{GN(O)}$. 
\begin{lemma}\label{GNO}
Let $O$ be a maximal order and let $K/F$ be a field extension of degree $p$ inside $B$. Then $K\subseteq H_{GN(O)}$ if and only if $K/F$ is an unramified extension and $B\cong \mathrm{M}_p(F)$.
\end{lemma}
\begin{proof}
First, we compute $GN(O)$:
$$ N_r(N_{B_{\mathfrak{p}}^\times}(O_\mathfrak{p}))= \begin{cases} F_{\mathfrak{p}}^\times, & \text{$B_\mathfrak{p}$ is a division algebra or $\mathfrak{p}\in\infty_F$}, \\
R_{\mathfrak{p}}^\times (F_{\mathfrak{p}}^\times)^p, & \text{otherwise}. \end{cases} $$
Hence $GN(O)\supseteq F^\times\widehat{R}^\times$, so $H_{GN(O)}/F$ is an unramified extension. This shows that if $K\subseteq H_{GN(O)}$, then $K$ must be everywhere unramified.

Now suppose $K\subseteq H_{GN(O)}$ and that there exists a prime $\mathfrak{p}$ with $B_\mathfrak{p}$ a division algebra. Then we have $N_m(K_\mathfrak{p})\supseteq F_{\mathfrak{p}}^\times$, so $K$ splits at $\mathfrak{p}$. But this is impossible: if $K$ splits at $\mathfrak{p}$ and $B_\mathfrak{p}$ is a division algebra, no embedding $K\hookrightarrow B$ exists. The converse direction is clear.
\end{proof}

At the end of this section, we introduce optimal embeddings and the selectivity sandwich.
\begin{definition} We say that $\mathrm{Gen} (O)$ is optimally selective for $S$ if there exists $O'\in\mathrm{Gen} O$ such that 
$\mathrm{Emb}(S,O')= \varnothing$. 
\end{definition}

We can refine Voight's \textit{selectivity sandwich} \cite[31.3.14]{voight2021quaternion}
\begin{align*}
F^\times_K N_r(\widehat{K}^\times)\leq F^\times N_r(\widehat{K}^\times)N_r(N_{\widehat{B}^\times}(\widehat{O}))\subseteq F^\times N_r(\widehat{E})\subseteq \widehat{F}^\times. \tag{*}
\end{align*}
If $K/F$ is galois, according to class field theory, we can deduce that 
$$[\widehat{F}^\times:F^\times_K N_r(\widehat{K}^\times)]=|\mathrm{Gal}(K/F)|=p.$$
When the degree is no longer 2, the first challenge we face is that $F^\times N_r(\widehat{E})$ is not necessarily a group. As a result, we cannot use the index of a group to determine whether a given inequality is strict or not. Nevertheless, we can still make the following argument. The following two lemmas are merely straightforward higher-degree generalizations of those in \cite[31.4]{voight2021quaternion}.

\begin{lemma}\label{1}
If $K/F$ is galois, then 
$$F^\times_K N_r(\widehat{K}^\times)= F^\times N_r(\widehat{K}^\times)N_r(N_{\widehat{B}^\times}(\widehat{O})) \ \ \ \text{if ane only if} \ \ \ K\subseteq H_{GN(O)} . $$ 
\end{lemma}

\begin{proof}
We know that the class field corresponding to $F^\times_K N_r(\widehat{K}^\times)$ is $K$, and the class field corresponding to $F^\times N_r(\widehat{K}^\times)N_r(N_{\widehat{B}^\times}(\widehat{O}))$ is $K\cap H_{GN(O)}$. This completes the proof.
\end{proof}

\begin{lemma}\label{2}
Let $O'\in \mathrm{Gen} O$, $[O']$ denote the type of $O'$, $[O']$ be represented by the class $B^\times \widehat{v} N_{\widehat{B}^\times}(\widehat{O})$. Then $\mathrm{Emb}(S,O')\neq \varnothing$ if and only if $N_r(\widehat{v})\in F^\times N_r(\widehat{E})$.
\end{lemma}

\begin{proof}
See \cite[Proposition 31.4.4]{voight2021quaternion}.
\end{proof}

\begin{corollary}\label{COR}
If $K\not \subseteq H_{GN(O)}$, $S$ can optimally embedding all order $O'\in \mathrm{Gen} O$.
\end{corollary}

\begin{proof}
If $K/F$ is not galois, then 
$$[\widehat{F}^\times:F^\times_K N_r(\widehat{K}^\times)]=|\mathrm{Gal}((K\cap F^{ab})/F)|=1.$$
We can directly obtain $F^\times N_r(\widehat{E})= \widehat{F}^\times$.

If $K/F$ is galois and $K\not \subseteq H_{GN(O)}$. According to Lemma \ref{1}, when $K\not \subseteq H_{GN(O)}$, we obtain $$F^\times_K N_r(\widehat{K}^\times)< F^\times N_r(\widehat{K}^\times)N_r(N_{\widehat{B}^\times}(\widehat{O}))$$ which in turn implies $F^\times N_r(\widehat{E})= \widehat{F}^\times$. Combining this with Lemma \ref{2}, we conclude the result.
\end{proof}

We are more interested in whether the converse holds; that is, if $K \subseteq H_{GN(O)}$, whether we can deduce that $\mathrm{Gen} O$ is optimally selective for $S$. For this purpose, we compute the local optimal embeddings.

\section{Local embedding numbers of degree 3}
To determine the global structure, we compute the local embedding numbers in this section.

Throughout this section, we assume that $p=3$, $F_\mathfrak{p}$ is a local field, and $R_\mathfrak{p}$ is a discrete valuation ring. Let $\mathfrak{p}$ be the maximal ideal of $R_\mathfrak{p}$, and fix a uniformizer $\pi$ such that $\mathfrak{p}=\pi R_\mathfrak{p}$. We denote by $v$ the discrete valuation associated with $\mathfrak{p}$.

\subsection{The normal case}
Let $K_\mathfrak{p}$ be a separable $F_\mathfrak{p}$-algebra of degree 3, $S_\mathfrak{p}$ an $R_\mathfrak{p}$-order in $K_\mathfrak{p}$, and $B_\mathfrak{p}$ a central simple algebra of degree $3$ over $F_\mathfrak{p}$. Since we assume $K \subseteq H_{GN(O)}$, we may describe $K_\mathfrak{p}$ as either $F_\mathfrak{p}\times F_\mathfrak{p}\times F_\mathfrak{p}$ or as an unramified field extension of $F_\mathfrak{p}$.

\begin{lemma}
Let $B_\mathfrak{p}$ be a division algebra and let $O_\mathfrak{p}\subset B_\mathfrak{p}$ be its valuation ring. If $K_\mathfrak{p}$ is a field and $S_\mathfrak{p}$ is integrally closed, then
$$m(S_\mathfrak{p}, O_\mathfrak{p}; N_{B_\mathfrak{p}^\times}(O_\mathfrak{p})) = 1.$$
If $K_\mathfrak{p}$ is not a field or $S_\mathfrak{p}$ is not integrally closed, then $\mathrm{Emb}(S_\mathfrak{p}, O_\mathfrak{p}) = \varnothing$.
\end{lemma}

\begin{proof}
If $B_\mathfrak{p}$ is a division $F_\mathfrak{p}$-algebra and $O_\mathfrak{p}\subset B_\mathfrak{p}$ is its valuation ring, then $O_\mathfrak{p}$ is the unique maximal $R_\mathfrak{p}$-order in $B_\mathfrak{p}$. Hence $N_{B_\mathfrak{p}^\times}(O_\mathfrak{p})=B_\mathfrak{p}^\times$, and $O_\mathfrak{p}$ is the integral closure of $R_\mathfrak{p}$ in $B_\mathfrak{p}$.

If $K_\mathfrak{p}$ is not a field, then there is no embedding $K_\mathfrak{p}\hookrightarrow B_\mathfrak{p}$. If $S_\mathfrak{p}$ is not integrally closed, then there is no optimal embedding.
\end{proof}

Now assume that $B_\mathfrak{p} = \mathrm{M}_3(R_\mathfrak{p})$, and choose an $R_\mathfrak{p}$-basis $\{e_1, e_2, e_3\}$ of $S_\mathfrak{p}$ with $e_1 = 1$. Since in this section we only consider the local case and have already assumed $B_\mathfrak{p} = \mathrm{M}_3(R_\mathfrak{p})$, throughout the remainder of this section we use the notation $B$ to denote a matrix rather than a central simple algebra. This convention applies only within this section. In this case, we obtain the following criterion.

\begin{lemma}\label{optimal}
Let $\varphi : S_\mathfrak{p} \hookrightarrow \mathrm{M}_3(R_\mathfrak{p})$ be an embedding, and write $\varphi(e_i)=A_i \in \mathrm{M}_3(R_\mathfrak{p})$. Let $\widetilde{A_i}$ denote the reduction of $A_i$ modulo $\mathfrak{p}$. The following conditions are equivalent.

(1) $\varphi$ is an optimal embedding.

(2) For all $1 \le i \le 3$ and for all $a_j \in R_\mathfrak{p}$ $(j \ne i)$,
$$\sum_{j \ne i} a_j A_j + A_i \notin \mathrm{M}_3(\pi R_\mathfrak{p}).$$

(2') For all $1 \le i \le 3$ and for all $\overline{a}_j \in R_\mathfrak{p}/\mathfrak{p}R_\mathfrak{p}$ $(j \ne i)$,
$$\sum_{j \ne i} \widetilde{a_j} \widetilde{A_j} + \widetilde{A_i} \ne 0.$$

(3) There exist $s_i, t_i \in \{1,2,3\}$ such that $$X := X_{s_i,t_i} =
\begin{pmatrix}
(A_1)_{s_1 t_1} & (A_2)_{s_1 t_1} & (A_3)_{s_1 t_1} \\
(A_1)_{s_2 t_2} & (A_2)_{s_2 t_2} & (A_3)_{s_2 t_2} \\
(A_1)_{s_3 t_3} & (A_2)_{s_3 t_3} & (A_3)_{s_3 t_3}
\end{pmatrix},$$
where $(A_i)_{s_j t_j}$ denotes the entry in the $s_j$-th row and $t_j$-th column of $A_i$, and $\det X \in R_\mathfrak{p}^\times$.

(3') There exist $s_i, t_i \in \{1,2,3\}$ such that $$\widetilde{X} := \widetilde{X_{s_i,t_i}} =
\begin{pmatrix}
(\widetilde{A_1})_{s_1 t_1} & (\widetilde{A_2})_{s_1 t_1} & (\widetilde{A_3})_{s_1 t_1} \\
(\widetilde{A_1})_{s_2 t_2} & (\widetilde{A_2})_{s_2 t_2} & (\widetilde{A_3})_{s_2 t_2} \\
(\widetilde{A_1})_{s_3 t_3} & (\widetilde{A_2})_{s_3 t_3} & (\widetilde{A_3})_{s_3 t_3}
\end{pmatrix},$$
where $(\widetilde{A_i})_{s_j t_j}$ denotes the entry in the $s_j$-th row and $t_j$-th column of $\widetilde{A_i}$, and $\det \widetilde{X} \ne 0$.
\end{lemma}

\begin{proof}
(2) $\Leftrightarrow$ (2') and (3) $\Leftrightarrow$ (3') are immediate.

(1) $\Rightarrow$ (2). Suppose that there exist $k_j \in R_\mathfrak{p}$ such that $$\sum k_j A_j + A_i \in \mathrm{M}_3(\pi R_\mathfrak{p}).$$ Then $$\sum \pi^{-1} k_j e_j + \pi^{-1} e_i \in K_\mathfrak{p} \setminus S_\mathfrak{p},$$ while $$\varphi\!\left(\sum \pi^{-1} k_j e_j + \pi^{-1} e_i\right) \in \mathrm{M}_3(R_\mathfrak{p}),$$ which is a contradiction.

(2) $\Rightarrow$ (1). Suppose that $\varphi$ is not an optimal embedding. Then there exist $k_i$ such that $$\varphi\!\left(\sum k_i e_i\right) \in \mathrm{M}_3(R_\mathfrak{p})$$ and $v(k_j) < 0$ for some $j$. Assume that $v(k_j) \le v(k_i)$ for all $i$. Then $$\varphi\!\left(\sum k_j^{-1} k_i e_i\right) \in \mathrm{M}_3(\pi R_\mathfrak{p}),$$ which contradicts (2).

(2') $\Leftrightarrow$ (3'). Condition (2') means that the matrices $\widetilde{A_i}$ are linearly independent over $R_\mathfrak{p}/\mathfrak{p}R_\mathfrak{p}$, while condition (3') means that one can select $3$ entries so that the resulting matrix has nonzero determinant. These two conditions are equivalent.
\end{proof}

Lemma \ref{optimal} gives an explicit criterion for determining whether a given embedding is optimal. The argument does not depend on the matrix size and carries over to $\mathrm{M}_n(R_\mathfrak{p})$ for any integer $n \ge 2$. In the case of $2\times 2$ matrices, the determinant condition in part (3) immediately implies
$$m(S_\mathfrak{p},\mathrm{M}_2(R_\mathfrak{p});\mathrm{GL}_2(R_\mathfrak{p}))=1.$$

The following example illustrates how this criterion works in the $2\times 2$ case.

\emph{Example} \cite[Proposition 30.5.3(a)]{voight2021quaternion}. We can directly enumerate the equivalent conditions in the case $n=2$. Let $\varphi(e_1)=E$ and
$$\varphi(e_2)=A=\begin{pmatrix}
a & b \\
c & d \\
\end{pmatrix}\in \mathrm{M}_2(R_\mathfrak{p}).$$
Swapping the positions of $s_i,t_i$ and $s_j,t_j$ corresponds to swapping two rows of $X$. This changes the determinant only by a sign and does not affect whether it lies in the unit group. Therefore, in what follows, we always require that $s_i,t_i$ be ordered according to lexicographic order. We compute
$$\det X_{11,12}=\begin{vmatrix}
1 & a \\
0 & b \\
\end{vmatrix}=b,
\det X_{11,21}=\begin{vmatrix}
1 & a \\
0 & c \\
\end{vmatrix}=c,
\det X_{11,22}=\begin{vmatrix}
1 & a \\
1 & d \\
\end{vmatrix}=d-a$$

$$
\det X_{12,21}=\begin{vmatrix}
0 & b \\
0 & c \\
\end{vmatrix}=0,
\det X_{12,22}=\begin{vmatrix}
0 & b \\
1 & d \\
\end{vmatrix}=-b,
\det X_{21,22}=\begin{vmatrix}
0 & c \\
1 & d \\
\end{vmatrix}=-c
$$
Thus, we conclude that $\varphi$ is optimal if and only if at least one of $b$, $c$, or $(d-a)$ lies in $R_\mathfrak{p}^\times$.

Furthermore, let $\phi:S_\mathfrak{p}\hookrightarrow \mathrm{End}_{R_\mathfrak{p}}(S_\mathfrak{p})=\mathrm{M}_2(R_\mathfrak{p})$ be the regular representation of $S_\mathfrak{p}$ on itself. In the basis $\{e_1,e_2\}$, we have:
\[\phi(e_1)=
\begin{pmatrix}
1 & 0 \\
0 & 1 \\
\end{pmatrix}
, \phi(e_2)=
\begin{pmatrix}
0  & \mathcal{N} \\
1  & \mathcal{T} 
\end{pmatrix}
\]
where $\mathcal{N}$ and $\mathcal{T}$ denote the coefficients of the minimal polynomial $f(x)=x^2-\mathcal{T}x-\mathcal{N}$ of $e_2$. We verify that $\varphi$ differs from the regular representation $\phi$ only by conjugation by a matrix in $\mathrm{GL}_2(R_\mathfrak{p})$. 

Let $u=(u_1,u_2)^T$, with $u_1,u_2\in R_\mathfrak{p}$, and let $U=(u,Au)$. We compute
$$\det U=cu_1^2+(d-a)u_1u_2-bu_2^2$$
Since at least one of $b$, $c$, or $(d-a)$ lies in $R_\mathfrak{p}^\times$, we can always find suitable $u_1,u_2\in R_\mathfrak{p}$ such that $\det U\in R_\mathfrak{p}^\times$ and $U\in \mathrm{GL}_2(R_\mathfrak{p})$. In this case, $U^{-1}AU$ is the regular representation of $e_2$ with respect to the basis $\{e_1,e_2\}$.

However, when $n=3$, this method breaks down. More precisely, in some cases, even if the embedding $\varphi$ is optimal, the matrix $U$ constructed as above always has $\det U \in \mathfrak{p}$. To treat the case $n=3$, we need to analyze similarity classes over the residue field $R_\mathfrak{p}/\mathfrak{p}R_\mathfrak{p}$.

\begin{lemma}\label{similar}
Every matrix $\widetilde{X}\in \mathrm{M}_3(R_\mathfrak{p}/\mathfrak{p}R_\mathfrak{p})$ is similar to exactly one matrix of the following forms:

(1) A scalar matrix $\widetilde{x}I$, with $\widetilde{x}\in R_\mathfrak{p}/\mathfrak{p}R_\mathfrak{p}$.

(2) A diagonal matrix of the form 
$$\begin{pmatrix}
\widetilde{x} & 0 & 0 \\
0 & \widetilde{x} & 0 \\
0 & 0 & \widetilde{y}
\end{pmatrix},\quad \text{with } \widetilde{x}\ne \widetilde{y}\in R_\mathfrak{p}/\mathfrak{p}R_\mathfrak{p}.$$

(3) A Jordan block of the form 
$$\begin{pmatrix}
\widetilde{x} & 0 & 0 \\
0 & \widetilde{x} & 1 \\
0 & 0 & \widetilde{x}
\end{pmatrix},\quad \widetilde{x}\in R_\mathfrak{p}/\mathfrak{p}R_\mathfrak{p}.$$

(4) A companion matrix 
$$\begin{pmatrix}
0 & 1 & 0 \\
0 & 0 & 1 \\
\widetilde{x} & \widetilde{y} & \widetilde{z}
\end{pmatrix},\quad \widetilde{x},\widetilde{y},\widetilde{z}\in R_\mathfrak{p}/\mathfrak{p}R_\mathfrak{p}.$$
\end{lemma}

\begin{proof}
It suffices to classify matrices according to the degree and factorization of their minimal polynomials.
\end{proof}

For convenience, throughout the remainder of this section we fix an $R_\mathfrak{p}$-basis $\{e_i\}$ of $S_\mathfrak{p}$ with $e_1 = 1$. Given an embedding $\varphi: S_\mathfrak{p} \hookrightarrow \mathrm{M}_3(R_\mathfrak{p})$, we write $A=\varphi(e_2)$ and $B=\varphi(e_3)$. Let 
$$\phi:S_\mathfrak{p}\hookrightarrow \mathrm{End}_{R_\mathfrak{p}}(S_\mathfrak{p})=\mathrm{M}_3(R_\mathfrak{p})$$ 
be the regular representation of $S_\mathfrak{p}$ on itself with respect to the basis $\{e_1,e_2,e_3\}$. We write $A_0=\phi(e_2)$ and $B_0=\phi(e_3)$. We denote by $\widetilde{A}, \widetilde{B}, \widetilde{A_0}, \widetilde{B_0} \in \mathrm{M}_3(R_\mathfrak{p}/\mathfrak{p}R_\mathfrak{p})$ the reductions of $A, B, A_0, B_0$ modulo $\mathfrak{p}$.

\begin{corollary} 
Let $\varphi :S_\mathfrak{p}\hookrightarrow \mathrm{M}_3(R_\mathfrak{p})$ is an optimal embedding. Then $\widetilde{A},\widetilde{B}$ cannot be similar to a scalar matrix $\widetilde{x}I$. 
\end{corollary}

\begin{theorem}
Let $\varphi : S_\mathfrak{p} \hookrightarrow \mathrm{M}_3(R_\mathfrak{p})$ be an optimal embedding. If one of $\widetilde{A}$ or $\widetilde{B}$ is similar to a matrix in case (2) or (4) in Lemma \ref{similar}, then there exists $U \in \mathrm{GL}_3(R_\mathfrak{p})$ such that 
$$U^{-1}AU = A_0,\quad U^{-1}BU = B_0.$$
\end{theorem}

\begin{proof}
Without loss of generality, assume that $\widetilde{A}$ is similar to one of the forms in Lemma \ref{similar}. We first reduce to the case where $\widetilde{A}$ itself has that form. 

Indeed, there exists $\widetilde{V} \in \mathrm{GL}_3(R_\mathfrak{p}/\mathfrak{p}R_\mathfrak{p})$ such that $\widetilde{V}^{-1}\widetilde{A}\widetilde{V}$ has the desired form. Let $V$ be a lift of $\widetilde{V}$ to $\mathrm{GL}_3(R_\mathfrak{p})$. Replacing $A$ by $V^{-1}AV$ (and $B$ by $V^{-1}BV$), we may assume that $\widetilde{A}$ already has the required form.

\medskip

\noindent
\textbf{Case (2).} Suppose that $\widetilde{A}$ is in case (2) in Lemma \ref{similar}. Since $\varphi$ is a homomorphism and $e_2,e_3$ commute, the matrices $\widetilde{A}$ and $\widetilde{B}$ commute. Hence
$$
\widetilde{A}=\begin{pmatrix}
\widetilde{x} & 0 & 0 \\
0 & \widetilde{x} & 0 \\
0 & 0 & \widetilde{y}
\end{pmatrix},
\quad
\widetilde{B}=\begin{pmatrix}
\widetilde{b_{11}} & \widetilde{b_{12}} & 0 \\
\widetilde{b_{21}} & \widetilde{b_{22}} & 0 \\
0 & 0 & \widetilde{b_{33}}
\end{pmatrix}.
$$

Let $u=(u_1,u_2,u_3)^T$ with $u_i \in R_\mathfrak{p}$, and set $U=(u,Au,Bu)$. Then
$$
\det \widetilde{U}
=\det \widetilde{X_{122233}}\widetilde{u_2}^2\widetilde{u_3}
+\det \widetilde{X_{112133}}\widetilde{u_1}^2\widetilde{u_3}
+\det \widetilde{X_{112233}}\widetilde{u_1}\widetilde{u_2}\widetilde{u_3},
$$
where
$$
\widetilde{X_{122233}}=-\widetilde{b_{12}}(\widetilde{x}-\widetilde{y}),\ 
\widetilde{X_{112133}}=\widetilde{b_{21}}(\widetilde{x}-\widetilde{y}),\ 
\widetilde{X_{112233}}=(\widetilde{b_{22}}-\widetilde{b_{11}})(\widetilde{x}-\widetilde{y}).
$$

If $\widetilde{b_{12}},\widetilde{b_{21}}$, and $\widetilde{b_{22}}-\widetilde{b_{11}}$ are all zero, then $\widetilde{B}$ is a linear combination of $\widetilde{A}$ and $\widetilde{E}$, contradicting the optimality of $\varphi$ by Lemma \ref{optimal}. Hence we can choose $u_i$ such that $\det \widetilde{U}\ne 0$. It follows that $U \in \mathrm{GL}_3(R_\mathfrak{p})$, and $U^{-1}AU,\; U^{-1}BU$ coincide with the regular representation.

\textbf{Case (4).} Suppose that $\widetilde{A}$ is in case (4). Then its minimal polynomial coincides with its characteristic polynomial. Since $\widetilde{A}$ and $\widetilde{B}$ commute, $\widetilde{B}$ is a polynomial in $\widetilde{A}$, say
$$
\widetilde{B}=\widetilde{k_1}\widetilde{E}+\widetilde{k_2}\widetilde{A}+\widetilde{k_3}\widetilde{A}^2.
$$
Thus
$$
\widetilde{A}=\begin{pmatrix}
0 & 1 & 0 \\
0 & 0 & 1 \\
\widetilde{x} & \widetilde{y} & \widetilde{z}
\end{pmatrix},
\quad
\widetilde{B}=\begin{pmatrix}
* & * & \widetilde{k_3} \\
* & * & \widetilde{k_2} \\
* & * & \widetilde{k_1}
\end{pmatrix}.
$$

Let $u=(0,0,1)^T$ and set $U=(u,Au,Bu)$. Then
$$
\det \widetilde{U}=\det \widetilde{X_{132333}}=-\widetilde{k_3}.
$$

By Lemma \ref{optimal}, $\widetilde{B}$ is not contained in the span of $\widetilde{E}$ and $\widetilde{A}$, hence $\widetilde{k_3}\ne 0$. Therefore $\det \widetilde{U}\ne 0$, and $U \in \mathrm{GL}_3(R_\mathfrak{p})$. Again, $U^{-1}AU,\; U^{-1}BU$ coincide with the regular representation.
\end{proof}

Next, we analyze case (3). If $\widetilde{A}$ is in case (3) of Lemma \ref{similar}, we have
$$\widetilde{A}=\begin{pmatrix}
\widetilde{x} & 0 & 0 \\
0 & \widetilde{x} & 1 \\
0 & 0 & \widetilde{x}
\end{pmatrix},
\quad
\widetilde{B}=\begin{pmatrix}
\widetilde{b_{11}} & 0 & \widetilde{b_{13}} \\
\widetilde{b_{21}} & \widetilde{b_{22}} & \widetilde{b_{23}} \\
0 & 0 & \widetilde{b_{22}}
\end{pmatrix}.$$

Let $u=(u_1,u_2,u_3)^T$ with $u_i \in R_\mathfrak{p}$, and set $U=(u,Au,Bu)$. We compute
$$\det \widetilde{U}=\det \widetilde{X_{112333}}\widetilde{u_1}\widetilde{u_3}^2+\det \widetilde{X_{132333}}\widetilde{u_3}^3,$$
where
$$\widetilde{X_{112333}}=\widetilde{b_{22}}-\widetilde{b_{11}},\quad \widetilde{X_{132333}}=-\widetilde{b_{13}}.$$

However, it may happen that $\widetilde{b_{22}}-\widetilde{b_{11}}=0$ and $\widetilde{b_{13}}=0$. To see this, we compute the remaining $\widetilde{X}$-terms:
$$\left\{
\begin{aligned}
-\widetilde{X_{111323}}=\widetilde{X_{132223}}&=\widetilde{b_{13}}, \\
-\widetilde{X_{112123}}=\widetilde{X_{212223}}=-\widetilde{X_{212333}}&=\widetilde{b_{21}}, \\
\widetilde{X_{112223}}&=\widetilde{b_{11}}-\widetilde{b_{22}}.
\end{aligned}
\right.$$

Thus it is possible that $\widetilde{b_{21}}\neq 0$ while $\widetilde{b_{22}}-\widetilde{b_{11}}=\widetilde{b_{13}}=0$. In this situation, even though $\varphi$ is optimal, we still have $\det \widetilde{U}=0$. This motivates the following definition:

\begin{definition}
Let $\varphi : S_\mathfrak{p}\hookrightarrow \mathrm{M}_3(R_\mathfrak{p})$ be an optimal embedding. Suppose that $\widetilde{A}$ is in case (3) of Lemma \ref{similar}, and that $\widetilde{B}=(\widetilde{b_{ij}})$ satisfies $\widetilde{b_{21}}\neq 0$ and $\widetilde{b_{22}}-\widetilde{b_{11}}=\widetilde{b_{13}}=0$. We say that $\varphi$ is a special optimal embedding.
\end{definition}

\begin{theorem}\label{nnormal}
Let $\varphi : S_\mathfrak{p}\hookrightarrow \mathrm{M}_3(R_\mathfrak{p})$ be an optimal embedding. If $\varphi$ is not a special optimal embedding, then there exists $U\in \mathrm{GL}_3(R_\mathfrak{p})$ such that
$$U^{-1}AU = A_0,\  U^{-1}BU = B_0.$$
\end{theorem}

To study these special optimal embeddings in detail, we treat separately the cases where $K_\mathfrak{p}$ is split and where $K_\mathfrak{p}$ is inert in the following two subsections.

\subsection{The split cases}
If $K_\mathfrak{p}$ is split (i.e. $K_\mathfrak{p}\cong F_\mathfrak{p}\times F_\mathfrak{p}\times F_\mathfrak{p}$), we can determine all $R_\mathfrak{p}$-orders $S_\mathfrak{p}$ in $K_\mathfrak{p}$ and find an $R_\mathfrak{p}$-basis for each $S_\mathfrak{p}$.

\begin{theorem}
Suppose that $S_\mathfrak{p}\subseteq R_\mathfrak{p}\times R_\mathfrak{p}\times R_\mathfrak{p}$ is an $R_\mathfrak{p}$-order. Then we have
$$
S_\mathfrak{p}=\{x=(x_1,x_2,x_3)\in R_\mathfrak{p}\times R_\mathfrak{p}\times R_\mathfrak{p}\mid x_1\equiv x_2 \ \mathrm{(mod}\ \pi^a),\ x_1 \equiv x_3 \ \mathrm{(mod}\ \pi^b)\}
$$
and $e_1=(1,1,1)$, $e_2=(0,\pi^a,0)$, $e_3=(0,0,\pi^b)$ form an $R_\mathfrak{p}$-basis.
\end{theorem}

We can compute 
$$
A_0=\begin{pmatrix}
0 & 0 & 0 \\
1 & \pi^a & 0 \\
0 & 0 & 0
\end{pmatrix},
B_0=\begin{pmatrix}
0 & 0 & 0 \\
0 & 0 & 0 \\
1 & 0 & \pi^b
\end{pmatrix}.
$$

\begin{lemma}\label{ssplit}
Let $\varphi : S_\mathfrak{p} \hookrightarrow \mathrm{M}_3(R_\mathfrak{p})$ be a special optimal embedding. Then there exists $U\in \mathrm{GL}_3(R_\mathfrak{p})$ such that $U^{-1}AU = A'_0$ and $U^{-1}BU = B'_0$, where 
$$A'_0=\begin{pmatrix}
0 & 0 & 0 \\
0 & 0 & 1 \\
0 & 0 & \pi^a
\end{pmatrix},
B'_0=\begin{pmatrix}
\pi^b & 0 & 0 \\
1 & 0 & 0 \\
0 & 0 & 0
\end{pmatrix}.$$
\end{lemma}
\begin{proof}
Since $\varphi$ is a special optimal embedding, without loss of generality we may assume that $\widetilde{A}$ is in case (3) of Lemma~\ref{similar}. According to Proposition~3.5 in \cite{Avni29072009}, we can conjugate $A$ by matrices in $R_\mathfrak{p}$ to bring it into the standard form $A_{r,s,t,m,d}$,
$$A_{r,s,t,m,d}=\begin{pmatrix} 0 & \pi^m  & 0 \\ 0 & 0 & 1 \\ r & s & t \end{pmatrix}+dI.$$
We know that the rank of $A$ is $1$, which implies that the first and second rows of $A_{r,s,t,m,d}$ are linearly dependent. This yields $d = 0$ and $m = +\infty$. Moreover, since the second and third rows are also linearly dependent, we obtain $r=s=0$. Finally, because $A$ and $A_{r,s,t,m,d}$ have the same trace, we obtain $t = \pi^a$.

Thus we may assume directly that $A=A'_0$. Let $B=(b_{ij})_{ij}$. Since $\varphi$ is special, we have $b_{21}\in R_\mathfrak{p}^\times$. Without loss of generality, we may assume $b_{21} = 1$; otherwise we can conjugate by $V=\operatorname{diag}(b_{21}^{-1},1,1)$. We know that $AB = 0$, which implies $b_{31}=b_{32}=b_{33}=0$. Similarly, $BA = 0$ implies $b_{22}+\pi^a b_{23}=0$. Since the rank of $B$ is $1$ and its trace is $\pi^b$, we can explicitly determine the entire matrix $B$:
$$B=\begin{pmatrix}
\pi^b+c\pi^a & -c\pi^a(\pi^b+c\pi^a) & c(\pi^b+c\pi^a) \\
1 & -c\pi^a & c \\
0 & 0 & 0
\end{pmatrix},\qquad c\in R_\mathfrak{p}.$$
Let 
$$U=\begin{pmatrix}
1 & c\pi^a & -c \\
0 & 1 & 0 \\
0 & 0 & 1
\end{pmatrix}.$$
Then we have $U^{-1}AU = A'_0$ and $U^{-1}BU = B'_0$.
\end{proof}

We claim that the embedding defined by $\varphi(e_2)=A'_0$ and $\varphi(e_3)=B'_0$ is a special optimal embedding if and only if $ab\neq 0$.

\begin{theorem}
Let $K_\mathfrak{p}\cong F_\mathfrak{p}\times F_\mathfrak{p}\times F_\mathfrak{p}$ and
$$
S_\mathfrak{p}=\{x=(x_1,x_2,x_3)\in R_\mathfrak{p}\times R_\mathfrak{p}\times R_\mathfrak{p}\mid x_1\equiv x_2 \ \mathrm{(mod}\ \pi^a),\ x_1 \equiv x_3 \ \mathrm{(mod}\ \pi^b)\}.
$$
The following statements hold.

(1) If $ab=0$, then
$$m(S_\mathfrak{p},\mathrm{M}_3(R_\mathfrak{p});\mathrm{GL}_3(R_\mathfrak{p}))=1.$$

(2) If $ab\neq 0$, then
$$m(S_\mathfrak{p},\mathrm{M}_3(R_\mathfrak{p});\mathrm{GL}_3(R_\mathfrak{p}))=2.$$
\end{theorem}

\begin{proof}
All our previous arguments show that there are at most two representatives in $m(S_\mathfrak{p},\mathrm{M}_3(R_\mathfrak{p});\mathrm{GL}_3(R_\mathfrak{p}))$: one is the regular representation $\phi$ of $S_\mathfrak{p}$, and the other is the embedding described in Lemma \ref{ssplit}. It remains to determine when these two representatives lie in the same orbit.

Assume that there exists $U\in \mathrm{GL}_3(R_\mathfrak{p})$ such that
$$U^{-1}A'_0 U = A_0,\qquad U^{-1}B'_0 U = B_0.$$
Let
$$
V=\begin{pmatrix}
1 & 0 & \pi^b \\
0 & 1 & 1 \\
1 & \pi^a & 0
\end{pmatrix}\in \mathrm{GL}_3(F_\mathfrak{p}).
$$
Then $V^{-1}A'_0V=A_0$ and $V^{-1}B'_0V=B_0$. It follows that $V^{-1}U$ lies in the common centralizer of $A_0$ and $B_0$. Write
$$
V^{-1}U=t_1E+t_2A_0+t_3B_0,
$$
so that
$$
U=V(t_1E+t_2A_0+t_3B_0)=\begin{pmatrix}
t_1 + \pi^{b} t_3 & 0 & \pi^{b}(t_1 + \pi^{b} t_3) \\
t_2 + t_3 & t_1 + \pi^{a} t_2 & t_1 + \pi^{b} t_3 \\
t_1 + \pi^{a} t_2 & \pi^{a}(t_1 + \pi^{a} t_2) & 0
\end{pmatrix}\in \mathrm{GL}_3(R_\mathfrak{p})
$$
and
$$
\det U=-(\pi^a+\pi^b)t_1(t_1 + \pi^{a} t_2)(t_1 + \pi^{b} t_3)\in R_\mathfrak{p}^\times.
$$
We know that $t_1 + \pi^{b} t_3$ and $t_1 + \pi^{a} t_2$ lie in $R_\mathfrak{p}$, so $\pi^{a} t_2-\pi^{b} t_3 \in R_\mathfrak{p}$. Since $\pi^b(t_2+t_3)\in R_\mathfrak{p}$, it follows that $\pi^{a} t_2+\pi^{b} t_2 \in R_\mathfrak{p}$.

If $a\neq b$ or $v(2)=0$, then $\pi^a t_2\in R_\mathfrak{p}$ and $t_1\in R_\mathfrak{p}$. Hence $\det U\in R_\mathfrak{p}^\times$ is impossible unless $ab=0$.

If $a=b$ and $v(2)>0$, then $2t_1\in R_\mathfrak{p}$. Hence
$$
\det U=-\pi^a(2t_1)(t_1 + \pi^{a} t_2)(t_1 + \pi^{b} t_3)\in R_\mathfrak{p}^\times
$$
is impossible unless $a=b=0$.

In any case, such a $U$ exists if and only if $ab=0$, which completes the proof of (2).

If $ab=0$ and $v(2)=0$, let $t_1=1$ and $t_2=t_3=0$. Then
$$
U=\begin{pmatrix}
1 & 0 & \pi^{b} \\
0 & 1 & 1 \\
1 & \pi^{a} & 0
\end{pmatrix}\in \mathrm{GL}_3(R_\mathfrak{p}).
$$

If $a=b=0$ and $v(2)>0$, let $t_1=t_2=t_3=\frac{1}{2}$. Then
$$
U=\begin{pmatrix}
1 & 0 & 1 \\
1 & 1 & 1 \\
1 & 1 & 0
\end{pmatrix}\in \mathrm{GL}_3(R_\mathfrak{p}).
$$
This completes the proof of (1).
\end{proof}

\begin{theorem}\label{msplit}
Let $K_\mathfrak{p}\cong F_\mathfrak{p}\times F_\mathfrak{p}\times F_\mathfrak{p}$, then $N_r(E_\mathfrak{p})=F_\mathfrak{p}^\times$.
\end{theorem}

\begin{proof}
$$N_r(E_\mathfrak{p})\supseteq N_r(K_\mathfrak{p}^\times)=F_\mathfrak{p}^\times.$$
\end{proof}

\subsection{The inert cases}
If $K_\mathfrak{p}$ is inert (i.e. $K_\mathfrak{p}/F_\mathfrak{p}$ is an unramified extension of local fields), in this subsection we denote the residue field of $K_\mathfrak{p}$ by $l$ and the residue field of $F_\mathfrak{p}$ by $k$. We can compute all $R_\mathfrak{p}$-orders $S_\mathfrak{p}$ in $K_\mathfrak{p}$ and determine an $R_\mathfrak{p}$-basis for each $S_\mathfrak{p}$. First, we know that the unique maximal order in $K_\mathfrak{p}$, namely the integral closure $\mathcal{O}_{K_\mathfrak{p}}$ of $R_\mathfrak{p}$ in $K_\mathfrak{p}$, admits an element $\alpha\in \mathcal{O}_{K_\mathfrak{p}}$ such that $\mathcal{O}_{K_\mathfrak{p}}=R_\mathfrak{p}[\alpha]$. Hence $\{1,\alpha,\alpha^2\}$ is an $R_\mathfrak{p}$-basis of $\mathcal{O}_{K_\mathfrak{p}}$.

A natural question is whether, for every $R_\mathfrak{p}$-order $S_\mathfrak{p}$, there always exists a basis of the form $\{1,\pi^a\alpha,\pi^b\alpha^2\}$. In general, for a fixed $\alpha$ the answer is negative; however, we can prove the following lemma.

\begin{lemma}
Let $S_\mathfrak{p}$ be an $R_\mathfrak{p}$-order. Then there exists $\alpha\in \mathcal{O}_{K_\mathfrak{p}}$ such that $\mathcal{O}_{K_\mathfrak{p}}=R_\mathfrak{p}[\alpha]$ and
$$
S_\mathfrak{p}=R_\mathfrak{p}+R_\mathfrak{p}\pi^a\alpha+R_\mathfrak{p}\pi^b\alpha^2.
$$
\end{lemma}

\begin{proof}
First, we determine $a$ and $\alpha$. Let
$$
a=\min \{ i\mid \pi^ix\in S_\mathfrak{p},\ \mathcal{O}_{K_\mathfrak{p}}=R_\mathfrak{p}[x]\}.
$$
This set is nonempty: since $F_\mathfrak{p}S_\mathfrak{p}=K_\mathfrak{p}$, for any generator $x$ of $\mathcal{O}_{K_\mathfrak{p}}$ there exists an $i$ such that $\pi^ix\in S_\mathfrak{p}$. Choose $\alpha$ such that $\pi^a\alpha\in S_\mathfrak{p}$ and $\mathcal{O}_{K_\mathfrak{p}}=R_\mathfrak{p}[\alpha]$. Let
$$
b=\min \{ j\mid \pi^j\alpha^2\in S_\mathfrak{p}\}.
$$
We know
$$
R_\mathfrak{p}+R_\mathfrak{p}\pi^a\alpha+R_\mathfrak{p}\pi^b\alpha^2\subseteq S_\mathfrak{p}
$$
and $\{1,\pi^a\alpha,\pi^b\alpha^2\}$ are linearly independent. Moreover, since $\{1,\alpha,\alpha^2\}$ is an $R_\mathfrak{p}$-basis of $\mathcal{O}_{K_\mathfrak{p}}$, for any $s\in S_\mathfrak{p}$ we have
$$
s = k_1+k_2\alpha+k_3\alpha^2,
$$
where $k_i\in R_\mathfrak{p}$. Therefore, it suffices to show $v(k_2)\geq a$ and $v(k_3)\geq b$. Without loss of generality, we assume $k_1 = 0$, $k_2 = \pi^{a'}$ and $k_3=u\pi^{b'}$, where $u\in R_\mathfrak{p}^\times$, $a',b'\in \mathbb{Z}\cup \{\infty\}$.

If $a'<b'$, then
$$
s=\pi^{a'}(\alpha+u\pi^{b'-a'}\alpha^2)
$$
and $(\alpha+u\pi^{b'-a'}\alpha^2)\equiv\alpha\ (\mathrm{mod}\ \pi)$, so $(\alpha+u\pi^{b'-a'}\alpha^2)$ is also a generator of $\mathcal{O}_{K_\mathfrak{p}}$, which implies that $a'\geq a$.

If $a'=b'$, then
$$
s=\pi^{a'}(\alpha+u\alpha^2),
$$
and $(\alpha+u\alpha^2)$ is a generator of $\mathcal{O}_{K_\mathfrak{p}}$, which implies that $a'\geq a$.

If $a'>b'$, then
$$
s=\pi^{b'}(\pi^{a'-b'}\alpha+u\alpha^2)
$$
and $(\pi^{a'-b'}\alpha+u\alpha^2)\equiv u\alpha^2\ (\mathrm{mod}\ \pi)$, so $(\pi^{a'-b'}\alpha+u\alpha^2)$ is also a generator of $\mathcal{O}_{K_\mathfrak{p}}$, which implies that $a'>b'\geq a$.

Thus, in any case we have $a'\geq a$. Moreover, since $u\pi^{b'}\alpha^2=s-\pi^{a'}\alpha\in S_\mathfrak{p}$, it follows that $b'\geq b$. This completes the proof.
\end{proof}
 
Let $\{1, \pi^a\alpha, \pi^b\alpha^2\}$ be an $R_\mathfrak{p}$-basis of $S_\mathfrak{p}$ with $a\leq b\leq 2a$. A computation shows that 
$$
\phi(\pi^a\alpha)=A_0=\begin{pmatrix}
0 & 0    &   a_0\pi^{a+b}   \\
1 &   0   & a_1\pi^{b}\\
 0 &   \pi^{2a-b}    &   a_2\pi^{a}   \\
\end{pmatrix},
\quad
\phi(\pi^b\alpha^2)=B_0=\begin{pmatrix}
0 & a_0\pi^{a+b}    &   a_0a_2\pi^{2b}   \\
0 &   a_1\pi^{b}   & (a_0+a_1a_2)\pi^{2b-a}\\
1 &   a_2\pi^{a}    &   (a_1+a_2^2)\pi^{b}   \\
\end{pmatrix}.
$$
where $f(x)=x^3-a_2x^2-a_1x-a_0$ is the minimal polynomial of $\alpha$.

Let $g(x)$ be the minimal polynomial of $A_0$. That is,
$$
g(x) = \pi^{3a}f(\pi^{-a}x)=x^3-a_2\pi^ax^2-a_1\pi^{2a}x-a_0\pi^{3a}.
$$
Then we have the following lemma:

\begin{lemma}\label{sinert}
Let $\varphi :S_\mathfrak{p}\hookrightarrow \mathrm{M}_3(R_\mathfrak{p})$ be a special optimal embedding. Then there exists $U\in \mathrm{GL}_3(R_\mathfrak{p})$ such that $U^{-1}AU = A'_0$ and $U^{-1}BU = B'_0$, where  
$$
A'_0=\begin{pmatrix}
a_2 \pi^{a} & \pi^{a+b} & 0 \\
0 & a_2 \pi^{a} & 1 \\
(a_0 + a_1 a_2)\pi^{2a-b} & (a_1 - a_2^2)\pi^{2a} & -a_2 \pi^{a}
\end{pmatrix},
$$
$$
B'_0={A'_0}^2\pi^{b-2a}=\begin{pmatrix}
a_2^{2}\pi^{b} & 2a_2\pi^{2b} & \pi^{2b-a} \\
(a_0+a_1a_2) & a_1\pi^{b} & 0 \\
0 & (a_0+a_1a_2)\pi^{a+b} & a_1\pi^{b}
\end{pmatrix}.
$$
\end{lemma}

\begin{proof}
Since $\varphi$ is a special optimal embedding, we may assume without loss of generality that $\widetilde{A}$ is in case (3) of Lemma \ref{similar}. According to Proposition 3.5 in \cite{Avni29072009}, we can conjugate $A$ by elements of $\mathrm{GL}_3(R_\mathfrak{p})$ to bring it into the standard form $A_{r,s,t,m,d}$:
$$
A_{r,s,t,m,d}=\begin{pmatrix} 0 & \pi^m  & 0 \\ 0 & 0 & 1 \\ r & s & t \end{pmatrix}+dI.
$$

Our goal is to show that all these parameters depend on $d$. Since $A$ and $\pi^a\alpha$ have the same minimal polynomial $g(x)$, this also coincides with the minimal polynomial of $A_{r,s,t,m,d}$. By comparing coefficients, we obtain
$$
\left\{
\begin{aligned}
&t=a_2\pi^a-3d \\
&s=a_1\pi^{2a}+2a_2\pi^ad-3d^2 \\
&r\pi^m=-g(d)
\end{aligned}
\right.
$$

Let $B = \varphi(\pi^b\alpha^2) = A^2\pi^{b-2a} = (b_{ij})_{ij}$. Since $\varphi$ is special, we have $b_{21}=r\pi^{b-2a}\in R_\mathfrak{p}^\times$. Hence we conclude that $v(r) = 2a-b$, which completely determines $r$ and $m$.
Consequently, $A$ can be written as

$$
A_d=\begin{pmatrix}
d & \pi^{v(g(d))-2a+b} & 0 \\
0 & d & 1 \\
\frac{-g(d)}{\pi^{v(g(d))}}\pi^{2a-b} & a_1\pi^{2a}+2a_2\pi^ad-3d^2 & a_2\pi^a-2d
\end{pmatrix}.
$$

Finally, we show that all matrices of the form $A_d$ are similar to $A_{a_2\pi^a}$. Let $h=h(d) = \dfrac{g(d)}{\pi^{v(g(d))}}\in R_\mathfrak{p}^\times$. Then
$$
U_d=\begin{pmatrix}
1 & (2 d \pi^{-2a+b} (d - a_2 \pi^{a}))h^{-1} & (\pi^{-2a+b} (d - a_2 \pi^{a}))h^{-1} \\
0 & -(a_0 + a_1 a_2)h^{-1} & 0 \\
0 & (-(a_0 + a_1 a_2)(d - a_2 \pi^{a}))h^{-1} & -(a_0 + a_1 a_2)h^{-1}
\end{pmatrix}
$$
and
$$
A_d=U_d^{-1}\begin{pmatrix}
a_2 \pi^{a} & \pi^{a+b} & 0 \\
0 & a_2 \pi^{a} & 1 \\
(a_0 + a_1 a_2)\pi^{2a-b} & (a_1 - a_2^2)\pi^{2a} & -a_2 \pi^{a}
\end{pmatrix}U_d.
$$

We claim that $\det U_d = (a_0+a_1a_2)^2 h^{-2} \in R_\mathfrak{p}^\times$. Indeed, since $\widetilde{f}(x)$ remains irreducible over the residue field $k$, we have $\widetilde{f}(\widetilde{a_2}) = -\widetilde{(a_0+a_1a_2)} \neq 0$, which forces $(a_0+a_1a_2)$ to be a unit.

It remains to verify that $\pi^{-2a+b} (d - a_2 \pi^{a})\in R_\mathfrak{p}$. But $B_d = A_d^2\pi^{b-2a}$ is the matrix corresponding to $\pi^b\alpha^2$, hence lies in $\mathrm{M}_3(R_\mathfrak{p})$; and $\pi^{-2a+b} (a_2 \pi^{a} - d)$ is precisely the entry in the second row and third column of $B_d$, so it belongs to $R_\mathfrak{p}$. This completes the proof of the lemma.
\end{proof}

\begin{theorem}\label{minert0}
Assume $K_\mathfrak{p}/F_\mathfrak{p}$ is an unramified extension of local fields and 
$$S_\mathfrak{p}=R_\mathfrak{p}+R_\mathfrak{p}\pi^a\alpha+R_\mathfrak{p}\pi^b\alpha^2,$$ 
where $\alpha\in \mathcal{O}_{K_\mathfrak{p}}$ satisfies $\mathcal{O}_{K_\mathfrak{p}}=R_\mathfrak{p}[\alpha]$ and $a\leq b \leq 2a$. The following statements hold.

(1) If $2a=b$, then $m(S_\mathfrak{p},\mathrm{M}_3(R_\mathfrak{p});\mathrm{GL}_3(R_\mathfrak{p}))=1$.

(2) If $2a\neq b$, then $m(S_\mathfrak{p},\mathrm{M}_3(R_\mathfrak{p});\mathrm{GL}_3(R_\mathfrak{p}))=2$.
\end{theorem}
\begin{proof}
As in the previous subsection, it suffices to determine whether the embedding $\varphi(e_2)=A'_0,\ \varphi(e_3)=B'_0$ given in Lemma~\ref{sinert} and the regular embedding $\phi$ lie in the same orbit.

Suppose there exists $U\in \mathrm{GL}_3(R_\mathfrak{p})$ such that $U^{-1}A'_0 U = A_0$. Let 
$$V=\begin{pmatrix}
0 & 0 & \pi^{2b-a} \\
0 & 1 & 0 \\
1 & -a_2\pi^a & a_1\pi^b
\end{pmatrix}\in \mathrm{GL}_3(F_\mathfrak{p}).$$
Then we have $V^{-1}A'_0 V = A_0$. It follows that $V^{-1}U$ lies in the centralizer of $A_0$. Since the characteristic polynomial of $A_0$ coincides with its minimal polynomial, the centralizer of $A_0$ consists precisely of polynomials in $A_0$. Hence we can write $V^{-1}U = t_1E + t_2A_0 + t_3A_0^2$, and therefore
$$U=\begin{psmallmatrix}
\pi^{a+b} t_3 & \pi^{a+b}(t_2 + a_2 \pi^{a} t_3) & \pi^{-a+2b}\bigl(t_1 + \pi^{a}(a_2 t_2 + a_1 \pi^{a} t_3 + a_2^2 \pi^{a} t_3)\bigr) \\
t_2 & t_1 + a_1 \pi^{2a} t_3 & \pi^{b}\bigl(a_0 \pi^{a} t_3 + a_1 (t_2 + a_2 \pi^{a} t_3)\bigr) \\
t_1 + \pi^{a}(-a_2 t_2 + a_1 \pi^{a} t_3) & \pi^{a}\bigl(-a_2 t_1 + \pi^{a}(a_1 t_2 + a_0 \pi^{a} t_3)\bigr) & \pi^{b}\bigl(a_1 t_1 + a_0 \pi^{a} t_2 + a_1^2 \pi^{2a} t_3\bigr)
\end{psmallmatrix}\in \mathrm{GL}_3(R_\mathfrak{p}).$$

If $a=0$, then we must have $b=0$ (otherwise $S_\mathfrak{p}$ would not be an order). In the case $a=b=0$, we can simply take $V$ to be the desired $U$. In the following analysis we always assume $0<a\leq b$. To analyze the determinant of $U$, a direct computation would be too complicated. Instead, we determine which entries of the matrix necessarily lie in the maximal ideal $\mathfrak{p}$. Since $U\in \mathrm{GL}_3(R_\mathfrak{p})$, all its entries belong to $R_\mathfrak{p}$; in particular $\pi^{a+b}t_3$, $t_2$, and $t_1 + a_1\pi^{2a}t_3$ are in $R_\mathfrak{p}$. Thus
\begin{align*}
U_{12}&=\pi^{a+b}(t_2 + a_2 \pi^{a} t_3)
      =\pi^{a+b}t_2 + a_2\pi^a(\pi^{a+b}t_3)\in \mathfrak{p},\\
U_{13}&=\pi^{-a+2b}\bigl(t_1 + \pi^{a}(a_2 t_2 + a_1 \pi^{a} t_3 + a_2^2 \pi^{a} t_3)\bigr)\\
      &=a_2\pi^{2b}t_2 + \pi^{-a+2b}(t_1 + a_1 \pi^{2a} t_3) + a_2^2\pi^b(\pi^{a+b}t_3)\in \mathfrak{p},\\
U_{33}&=\pi^{b}\bigl(a_1 t_1 + a_0 \pi^{a} t_2 + a_1^2 \pi^{2a} t_3\bigr)
      =a_0\pi^{a+b}t_2 + a_1\pi^b(t_1 + a_1 \pi^{2a} t_3)\in \mathfrak{p}.
\end{align*}
For $\det U$ to be a unit in $R_\mathfrak{p}$, we must have $U_{11},U_{23}\in R_\mathfrak{p}^\times$, and consequently $U_{32}\in R_\mathfrak{p}^\times$.
Now
$$U_{32}=\pi^{a}\bigl(-a_2 t_1 + \pi^{a}(a_1 t_2 + a_0 \pi^{a} t_3)\bigr)
      =a_1\pi^{2a}t_2 - a_2\pi^a t_1 + a_0\pi^{3a}t_3.$$
If $U_{11}\in R_\mathfrak{p}^\times$, then $v(\pi^{a+b}t_3)=0$, so $v(t_3)=-a-b$. Hence $v(a_1\pi^{2a}t_3)\ge a-b$. From $t_1 + a_1\pi^{2a}t_3\in R_\mathfrak{p}$ we obtain $v(t_1)\ge a-b$, and therefore $v(a_2\pi^a t_1)\ge 2a-b$. Moreover, $v(a_0\pi^{3a}t_3)=2a-b$. This forces $v(U_{32})\ge 2a-b$. Since $U_{32}$ must be a unit, the only possibility is $2a=b$.

When $2a=b$, choose $t_1=-a_1\pi^{-a}$, $t_2=0$, $t_3=\pi^{-3a}$. Then
$$U=\begin{pmatrix}
1 & a_2 \pi^{a} & a_2^2 \pi^{2a} \\
0 & 0 & a_0 + a_1 a_2 \\
0 & a_0 + a_1 a_2 & 0
\end{pmatrix}\in \mathrm{GL}_3(R_\mathfrak{p}),$$
which completes the proof of the theorem.
\end{proof}

\begin{theorem}\label{minert}
Assume $K_\mathfrak{p}/F_\mathfrak{p}$ is an unramified extension of local fields. Let 
$$E_\mathfrak{p}=\{U\in \mathrm{GL}_3(F_\mathfrak{p})\mid U^{-1}\phi(K_\mathfrak{p})U\cap \mathrm{M}_3(R_\mathfrak{p})= U^{-1}\phi(S_\mathfrak{p})U\},$$ 
where $\phi$ is the regular representation of $S_\mathfrak{p}$ on itself with respect to the basis $\{1, \pi^a\alpha, \pi^b\alpha^2\}$. Then 
$$N_r(E_\mathfrak{p})=N_r(K^\times_\mathfrak{p})\cup \sqrt{\mathrm{disc}(S_\mathfrak{p})}N_r(K^\times_\mathfrak{p}).$$
\end{theorem}
\begin{proof}
By Theorem \ref{minert0}, we know that $E_\mathfrak{p}$ has at most two orbits, and a generator of the nontrivial orbit is given by  
$$V=\begin{pmatrix}
0 & 0 & \pi^{2b-a} \\
0 & 1 & 0 \\
1 & -a_2\pi^a & a_1\pi^b
\end{pmatrix}\in \mathrm{GL}_3(F_\mathfrak{p}).$$
Then we have $V^{-1}A'_0 V = A_0$ and $E_\mathfrak{p}=K_\mathfrak{p}^\times\mathrm{GL}_3(R_\mathfrak{p})\cup V^{-1}K_\mathfrak{p}^\times\mathrm{GL}_3(R_\mathfrak{p})$. Since 
$$(2b-a)+(a+b)\equiv 0 \ (\mathrm{mod}\ 3),$$
it follows that 
$$N_r(E_\mathfrak{p})=(F^\times_\mathfrak{p})^3\cup \pi^{a+b}(F^\times_\mathfrak{p})^3.$$
Moreover, we know that the parameter $(a+b)$ does not depend on the choice of $\alpha$, but only on $S_\mathfrak{p}$. This follows by considering the discriminant of $\mathrm{disc}(S_\mathfrak{p})$ over $R_\mathfrak{p}$: 
$$\mathrm{disc}(S_\mathfrak{p})=\det(\mathrm{diag}(1,\pi^a,\pi^b))^2\mathrm{disc}(\mathcal{O}_{K_\mathfrak{p}})=\mathfrak{p}^{2(a+b)},$$
which implies 
$$N_r(E_\mathfrak{p})=N_r(K^\times_\mathfrak{p})\cup \sqrt{\mathrm{disc}(S_\mathfrak{p})}N_r(K^\times_\mathfrak{p}).$$
\end{proof}

\begin{proposition}\label{EE}
Let $O'_\mathfrak{p}=W^{-1}\mathrm{M}_3(R_\mathfrak{p})W$ be another maximal order and $\psi:S_\mathfrak{p}\hookrightarrow O'_\mathfrak{p}$ be an other optimal embedding. By the Skolem–Noether theorem, $\phi(K_\mathfrak{p})$ and $\psi(K_\mathfrak{p})$ differ only by a conjugation, so let $\psi(K_\mathfrak{p})= \mathcal{W}^{-1}\phi(K_\mathfrak{p})\mathcal{W}$, where $\mathcal{W}\in \mathrm{GL}_3(F_\mathfrak{p})$. Then for 
$$E'_\mathfrak{p}=\{U'\in \mathrm{GL}_3(F_\mathfrak{p})\mid U'^{-1}\psi(K_\mathfrak{p})U'\cap O'_\mathfrak{p}= U'^{-1}\psi(S_\mathfrak{p})U'\}.$$
We have $N_r(E'_\mathfrak{p})=N_r(W\mathcal{W}^{-1})N_r(E_\mathfrak{p})$.
\end{proposition}
\begin{proof}
For simplicity, we use $(*)^U$ to denote $U^{-1} (*) U$. Let $U'\in E'_\mathfrak{p}$, we have
$$\psi(K_\mathfrak{p})^{U'}\cap O'_\mathfrak{p}=\psi(S_\mathfrak{p})^{U'},$$
so
$$\phi(K_\mathfrak{p})^{U'\mathcal{W}}\cap O'_\mathfrak{p}=\phi(S_\mathfrak{p})^{U'\mathcal{W}},$$
and
$$\phi(K_\mathfrak{p})^{W^{-1}U'\mathcal{W}}\cap \mathrm{M}_3(R_\mathfrak{p})=\phi(S_\mathfrak{p})^{W^{-1}U'\mathcal{W}}.$$
Thus, we obtain a one-to-one correspondence between $E_\mathfrak{p}$ and $E'_\mathfrak{p}$:
\begin{align*}
E'_\mathfrak{p}&\longrightarrow E_\mathfrak{p}\\
U'&\longrightarrow W^{-1}U'\mathcal{W}\\
WU\mathcal{W}^{-1}&\longleftarrow U
\end{align*}
Hence we have $N_r(E'_\mathfrak{p})=N_r(W\mathcal{W}^{-1})N_r(E_\mathfrak{p})$.
\end{proof}

\section{Main theorem}
By Theorems~\ref{msplit} and~\ref{minert}, we have determined the structure of $N_r(E_\mathfrak{p})$ for every place $\mathfrak{p}$. To describe the global structure $N_r(\widehat{E})$, we introduce the following map:
$$\rho: \mathrm{Cl}(R)\cong \mathrm{Gal}(H/F)\longrightarrow \mathrm{Gal}(K/F),$$
where $H$ is the Hilbert class field of $F$ and the arrow denotes the restriction map. In terms of ideles, for each prime ideal $\mathfrak{p}$, the element $\rho([\mathfrak{p}])\in\mathrm{Gal}(K/F)$ corresponds to the coset $e_{\mathfrak{p}}F^\times N_r(\widehat{K}^\times)$, where
$$e_{\mathfrak{p}}=(1,1,\dots,\pi_{\mathfrak{p}},\dots).$$

\begin{definition}
Let $K/F$ be a degree 3 unramified extension of number fields, let $R$ be the ring of integers of $F$, and let $S$ be an $R$-order in $K$. Let $\mathrm{disc}(S)$ be the discriminant ideal of $S$ over $R$; then $\mathrm{disc}(S)$ is the square of an ideal in $R$, and we denote this ideal by $\sqrt{\mathrm{disc}(S)}$. We define
$$D(S) = \{ \rho([\mathfrak{p}^k]) : \mathfrak{p} \text{ is a prime ideal of } R,\ \mathfrak{p}^k \mid \sqrt{\mathrm{disc}(S)},\ \mathfrak{p}^{k+1} \nmid \sqrt{\mathrm{disc}(S)} \},$$
and call $D(S)$ the selectivity set of $S$.
\end{definition}

\begin{theorem}\label{main}
Let $B$ be a central simple algebra of degree $3$ over a number field $F$ with ring of integers $R$, let $K/F$ be a degree 3 extension inside $B$, and let $S\subset K$ be an $R$-order. Suppose $O\subset B$ is a maximal $R$-order and there is an optimal embedding $S\hookrightarrow O$. Then $S$ cannot optimally embedding into every maximal order of $B$ if and only if $K\subseteq H_{GN(O)}$ and $D(S)\neq \mathrm{Gal}(K/F)$. In this case $S$ is optimally embedded into exactly $\dfrac{|D(S)|}{3}$ of the types $[O']$ in $\mathrm{Typ}(O)$.
\end{theorem}

\begin{proof}
Recall that Corollary~\ref{COR} tells us that when $K \not\subseteq H_{GN(O)}$, $\mathrm{Gen}(O)$ is not optimally selective. Therefore, we only consider the case $K \subseteq H_{GN(O)}$. By Lemma~\ref{GNO}, we know $B\cong \mathrm{M}_3(F)$; without loss of generality, we assume that $B= \mathrm{M}_3(F)$. In this case, we aim to obtain a more refined description of the structure of $F^\times N_r(\widehat{E})$ in the selectivity sandwich
\begin{align*}
F^\times N_r(\widehat{K}^\times)= F^\times N_r(\widehat{K}^\times)N_r(N_{\widehat{B}^\times}(\widehat{O}))\subseteq F^\times N_r(\widehat{E})\subseteq \widehat{F}^\times \tag{*}
\end{align*}

We know $[\widehat{F}^\times:F^\times N_r(\widehat{K}^\times)]=3$, and the set $F^\times N_r(\widehat{E})$ is stable under multiplication by $F^\times N_r(\widehat{K}^\times)$, hence is a union of its cosets in $\widehat{F}^\times$. We fix a canonical optimal embedding $\widehat{\phi}:\widehat{S}\rightarrow \mathrm{M}_3(\widehat{R})$ such that, at every inert prime ideal $\mathfrak{p}$, $\phi_\mathfrak{p}: S_\mathfrak{p}\rightarrow \mathrm{M}_3(R_\mathfrak{p})$ is of the form of the regular embedding described in Theorem~\ref{minert}. We define 
$$\widehat{E_0}:= \{\widehat{U}\in \mathrm{GL}_3(\widehat{F})\mid \widehat{U}^{-1}\widehat{\phi}(\widehat{K})\widehat{U}\cap \mathrm{M}_3(\widehat{R})= \widehat{U}^{-1}\widehat{\phi}(\widehat{S})\widehat{U}\}.$$

It suffices to show that the elements in the selectivity set $D(S)$ of $S$ determine which cosets of $F^\times N_r(\widehat{K}^\times)$ constitute $N_r(\widehat{E})$. Write the prime ideal factorization of $\mathrm{disc}(S)$ as 
$$\mathrm{disc}(S)=\prod\mathfrak{p}_i^{2k_i}.$$
Therefore, we have $\mathrm{disc}(S_{\mathfrak{p}_i})=\mathfrak{p}_i^{2k_i}$. By Theorem~\ref{minert}, this implies that for every inert prime $\mathfrak{p}_i$,
$$N_r({E_0}_{\mathfrak{p}_i})=N_r(K^\times_{\mathfrak{p}_i})\cup \mathfrak{p}_i^{k_i}N_r(K^\times_{\mathfrak{p}_i}),$$
and by Theorem~\ref{msplit} we have 
$$N_r(\widehat{E_0})=N_r(\widehat{K}^\times)\bigcup e_{\mathfrak{p}_i}^{k_i}N_r(\widehat{K}^\times).$$
By class field theory, for a prime ideal $\mathfrak{q}$ of $F$ that splits completely in $K$, $\rho([\mathfrak{q}])$ is always the identity element. Therefore, it is reasonable to consider the map $\rho$ only on inert primes. Consequently,
$$D(S)\cong F^\times N_r(\widehat{E_0})/F^\times N_r(\widehat{K}^\times)\subseteq  \widehat{F}^\times/F^\times N_r(\widehat{K}^\times)\cong \mathrm{Gal}(K/F).$$
By Proposition~\ref{EE}, we have $\widehat{E}=\widehat{v}\widehat{E_0}$ for some $\widehat{v}\in \widehat{F}$. This merely translates the cosets in $\widehat{F}^\times/F^\times N_r(\widehat{K}^\times)$, so the number of elements in $N_r(\widehat{E})/F^\times N_r(\widehat{K}^\times)$ remains unchanged.

Thus, we have shown that when $K \subseteq H_{GN(O)}$, $D(S)=\mathrm{Gal}(K/F)$ if and only if $S$ can be optimally embedded into every maximal order. Moreover, consider the map
$$\rho':\mathrm{Typ}(O)\cong \mathrm{Gal}(H_{GN(O)}/F)\longrightarrow \mathrm{Gal}(K/F).$$
For a type $[O']\in \mathrm{Typ}(O)$, we have $\rho'([O'])\in \widehat{v}D(S)$ if and only if $S$ can be optimally embedded into $O'$, where the translation by $\widehat{v}$ is realized via the Artin map on ideles. Therefore, $S$ is optimally embedded into exactly $\dfrac{|\widehat{v}D(S)|}{|\mathrm{Gal}(K/F)|}=\dfrac{|D(S)|}{3}$ of the types $[O']$ in $\mathrm{Typ}(O)$.
\end{proof}

We compute some examples to illustrate how Theorem~\ref{main} is applied in practice.

\emph{Example}. Let $F=\mathbb{Q}(\sqrt{-23})$, $R=\mathbb{Z}[\frac{\sqrt{-23}+1}{2}]$; then $\mathrm{Cl}(R)=\langle [\mathfrak{p}] \rangle$, where $\mathfrak{p}=(2,\frac{\sqrt{-23}+1}{2})$. Let $K=F(\alpha)$, where $\alpha$ is a root of $f(x)=x^3-x^2+1$. We compute $\mathrm{Gal}(K/F)=\langle\sigma\rangle$, where
$$\sigma(\alpha)=\frac{1-\alpha}{2}+\frac{\sqrt{-23}}{46}(3-7\alpha-2\alpha^2).$$
$K$ is the Hilbert class field of $F$ and is thus unramified everywhere. The Artin map is
\begin{align*}
\mathrm{Art}_{K/F}:\mathrm{Cl}(R)&\longrightarrow \mathrm{Gal}(K/F)\\
[\mathfrak{p}]&\longrightarrow \sigma .
\end{align*}
Let $B=\mathrm{M}_3(F)$, $O=\mathrm{M}_3(R)$; then $K=H_{GN(O)}$ and $\mathrm{Typ} (O)=\{[O_1],[O_2],[O_3]\}$, where $O_1=O$,
$$O_2=\begin{pmatrix}
R & \mathfrak{p} & \mathfrak{p} \\
\mathfrak{p}^{-1} & R & R \\
\mathfrak{p}^{-1} & R & R
\end{pmatrix}, 
O_3=\begin{pmatrix}
R & \mathfrak{p}^2 & \mathfrak{p}^2 \\
\mathfrak{p}^{-2} & R & R \\
\mathfrak{p}^{-2} & R & R
\end{pmatrix}.$$
We aim to find some $R$-orders in $K$ that have different selectivity sets $D(S_i)$. Let $\beta=\frac{\alpha^2+15\alpha+10}{\sqrt{-23}}\in K$. We have
$$\beta^3+2\sqrt{-23}\beta^2-29\beta+\sqrt{-23}=0,$$
so $\beta\in \mathcal{O}_K$. By computing the discriminant, we see that $\{1, \alpha, \beta\}$ indeed forms an $R$-basis of $\mathcal{O}_K$. We show that $S_1=\mathcal{O}_K$ admits an optimal embedding into $O_1$, but does not admit an optimal embedding into $O_2$ or $O_3$. Let $\varphi:S_1\rightarrow O_1$, where

$$\varphi(\alpha)=\begin{pmatrix}
0 & -10 & 7\sqrt{-23}\\
1 & -15 & 10\sqrt{-23}\\
0 & \sqrt{-23} & 16
\end{pmatrix},
\varphi(\beta)=\begin{pmatrix}
0 & 7\sqrt{-23} & 117\\
0 & 10\sqrt{-23} & 167\\
1 & 16 & -12\sqrt{-23}
\end{pmatrix}.$$
Since $S_1$ is a maximal $R$-order and $K/F$ is unramified everywhere, we have $\mathrm{disc}(S_1)=R$ and $D(S_1)=\{1\}$. This implies $[\widehat{F}^\times:F^\times N_r(\widehat{E})]=3$, so only one third of the types in $\mathrm{Typ}(O)$ admit an optimal embedding. Therefore, $S_1$ can be optimally embedded only into $O_1$.

Now let $S_2=R+\mathfrak{p}\mathcal{O}_K$ and let $\varphi:K\hookrightarrow B$ be an embedding. One checks that $$\varphi(K)\cap O_2=\varphi(S_2),$$ hence $\varphi:S_2\hookrightarrow O_2$ is an optimal embedding. We compute $\mathrm{disc}(S_2)=\mathfrak{p}^4$ and $D(S_2)=\{1,\sigma^2\}$. Let $\pi_\mathfrak{p}=2$ and compute $W_\mathfrak{p}$ and $\mathcal{W}_\mathfrak{p}$ as in Proposition~\ref{EE}. We have 
$${O_2}_\mathfrak{p}=\begin{pmatrix}
2 & 0 & 0\\
0 & 1 & 0\\
0 & 0 & 1
\end{pmatrix}\mathrm{M}_3(R_\mathfrak{p})\begin{pmatrix}
\frac{1}{2} & 0 & 0\\
0 & 1 & 0\\
0 & 0 & 1
\end{pmatrix},$$
so $W_\mathfrak{p}=\mathrm{diag}\{\frac{1}{2},1,1\}$. We know ${S_2}_\mathfrak{p}=R_\mathfrak{p}+2\alpha R_\mathfrak{p}+2\alpha^2 R_\mathfrak{p}$ and 
$$\phi_\mathfrak{p}(2\alpha)=\begin{pmatrix}
0 & 0 & -4\\
1 & 0 & 0\\
0 & 2 & 2
\end{pmatrix},\quad
\phi_\mathfrak{p}(2\alpha^2)=\begin{pmatrix}
0 & -4 & -4\\
0 & 0 & -2\\
1 & 2 & 2
\end{pmatrix}.$$
Let 
$$\mathcal{W}_\mathfrak{p}^{-1}=\begin{pmatrix}
1 & 0 & -20\\
0 & 2 & -30\\
0 & 0 & 2\sqrt{-23}
\end{pmatrix};$$ 
we have $\varphi(\alpha)=\mathcal{W}_\mathfrak{p}^{-1}\phi_\mathfrak{p}(\alpha)\mathcal{W}_\mathfrak{p}$ and $N_r(W_\mathfrak{p}\mathcal{W}_\mathfrak{p}^{-1})=2\sqrt{-23}$. Set
$$\widehat{v}_2=(1,1,\dots,2\sqrt{-23},\dots).$$
Then $\widehat{v}_2D(S_2)=\{\sigma,1\}$. Finally, computing the map $\rho'$ gives $\rho'([O_2])=1$, $\rho'([O_1])=\sigma$, and we conclude that $S_2$ can only be optimally embedded into $O_2$ and $O_1$.

This example illustrates the difference between Theorem~\ref{main} and the corresponding result in \cite[Theorem 3.3]{Linowitz2012}. For $S_2$, Linowitz and Shemanske show that $S_2$ does not satisfy selectivity, i.e., $S_2$ can be embedded into every maximal order. However, we have shown that $S_2$ cannot be optimally embedded into all maximal orders.

\section*{Acknowledgments}
The author is deeply grateful to his advisor, Professor Fei Xu, for many valuable suggestions and comments on this paper.

\bibliography{main}
 
\end{document}